\documentclass[a4paper,12pt]{article}
\usepackage[utf8]{inputenc}

\usepackage{mathrsfs}
\usepackage{graphicx}
\usepackage{enumerate}
\usepackage{multicol}
\usepackage{color}
\usepackage{amsmath,amssymb,amscd}
\usepackage{amsthm}
\usepackage{mathabx}

\usepackage{pdfpages}
\usepackage{hyperref}

\usepackage{times}
\usepackage[varg]{txfonts}

\usepackage[top=1in, bottom=1in, left=0.75in, right=0.75in]{geometry}

\newcommand{\R}{\mathbb{R}}
\newcommand{\C}{\mathbb{C}}
\newcommand{\N}{\mathbb{N}}

\newtheorem{thm}{Theorem}[section] 
\theoremstyle{definition} 
\theoremstyle{remark} 

\newtheorem{prop}{Proposition}
\newtheorem*{ackn}{Acknowledgements}

\newcommand{\factor}[2]{{\raisebox{0.em}{$#1$}\left/\raisebox{-.1em}{$#2$}\right.}}

\begin{document}
 \title{Generalized derivations and generalized\\ exponential monomials on hypergroups}
 \author{{\.Z}ywilla Fechner, Eszter Gselmann and L\'{a}szl\'{o} Sz\'ekelyhidi}
%

\maketitle

\begin{abstract}
In one of our former papers {\it Endomorphisms of the measure algebra of commutative hypergroups (\href{https://arxiv.org/abs/2204.07499}{arXiv:2204.07499})} we considered exponential monomials on hypergroups and  higher order derivations of the corresponding measure algebra. 
Continuing with this, we are now looking for the connection between the generalized exponential polynomials of a commutative hypergroup and the higher order derivations of the corresponding measure algebra. 
\end{abstract}


\section{Introduction}
\label{sec:introduction}
In this paragraph we are going to discuss moment functions and exponential monomials in the hypergroup settings. Concerning hypergroups, we will follow the notation and  terminology of the monograph \cite{BloHey95}. More about moment functions and exponentials on hypergroups one can find in \cite{Sze13} and references therein. 

Given a commutative hypergroup $X$ we denote by $\mathscr C(X)$ the space of all continuous complex valued functions on $X$. Equipped with the pointwise linear operations (addition of functions and multiplication by scalars) and with the topology of uniform convergence on compact sets $\mathscr C(X)$ is a locally convex topological vector space. The space $\mathscr C(X)$ equipped with pointwise multiplication is also a topological ring which will be utilized in the sequel. The topological dual of $\mathscr C(X)$ as a topological vector space can be identified with the space $\mathscr M_c(X)$ of all compactly supported complex Borel measures with the pairing
$$
\langle \mu,f\rangle=\int_X f\,d\mu.
$$
In fact, $\mathscr M_c(X)$ itself is a commutative algebra equipped with the pointwise linear operations and with the convolution of measures defined by
$$
\langle \mu*\nu,f\rangle=\int_X \int_X f(x*y)\,d\mu(x)\,d\nu(y)
$$
for each $\mu,\nu$ in $\mathscr M_c(X)$ and $f$ in $\mathscr C(X)$. We shall call $\mathscr M_c(X)$ the {\it measure algebra} of $X$. The measure algebra bears its natural weak*-topology, and its topological dual can be identified with $\mathscr C(X)$: every weak*-continuous linear functional $\Lambda$ on $\mathscr M_c(X)$ arises from a continuous function $f$ in $\mathscr C(X)$ in the natural way
$$
\Lambda(\mu)=\int_X f\,d\mu
$$
whenever $\mu$ is in $\mathscr M_c(X)$ (see e.g. \cite[Theorem 3.43]{MR3185617}).

Besides the algebra structure $\mathscr M_c(X)$ can be considered as a module over the ring $\mathscr C(X)$ via the following action: for each $\mu$ in $\mathscr M_c(X)$ and $\varphi, f$ in $\mathscr C(X)$ we define
$$
\langle \varphi \cdot \mu,f\rangle=\int_X f\cdot \varphi\,d\mu.
$$
It is easy to check that $\mathscr M_c(X)$ is a module over the ring $\mathscr C(X)$. We simple write $\varphi \mu$ for $\varphi\cdot \mu$. We shall call {\it module homomorphism} a mapping $F:\mathscr M_c(X)\to\mathscr M_c(X)$ if it satisfies the equations
$$
F(\mu+\nu)=F(\mu)+F(\nu),\qquad \text{and} \qquad F(\varphi \mu)=\varphi F(\mu)
$$
for each $\mu,\nu$ in $\mathscr M_c(X)$ and $f$ in $\mathscr C(X)$. We note that every module homomorphism of $\mathscr M_c(X)$ is also a linear homomorphism, as multiplication by scalars is the same as multiplication by constant functions. Module homomorphisms of the measure algebra of a commutative hypergroup are described in  the recent paper of the present authors \cite{FecGseSze22b}.

In this paper we use  multi-index notation. We recall that, besides the usual vector-notation for the basic operations we use the following notation: for every  multi-indices $\alpha=(\alpha_1,\alpha_2,\dots,\alpha_r)$ and $\beta=(\beta_1,\beta_2,\dots,\beta_r)$ in $\N^r$, we shall write
$\alpha\leq \beta\enskip\text{whenever}\enskip \alpha_i\leq \beta_i\enskip\text{for} \enskip i=1,2,\dots r$, and the symbol  $\alpha<\beta$ means that  $\alpha\leq\beta$ and $\alpha\ne \beta$. Further, we use the notations
$$
|\alpha|=\alpha_1+\alpha_2+\cdots+\alpha_r,\hskip1cm \alpha!=\alpha_1!\cdot \alpha_2!\cdots\alpha_r!,
$$
$$
\binom{\alpha}{\beta}=\frac{\alpha!}{\beta! \cdot (\alpha-\beta)!},\hskip1cm x^{\alpha}=x_1^{\alpha_1}\cdot x_2^{\alpha_2}\cdot\dots\cdot x_r^{\alpha_r}.
$$
If there is no misunderstanding, the zero of $\N^r$ will be denoted by $0$ instead of $(0,0,\dots,0)$.

Let $X$ be a commutative hypergroup and $r$ a positive integer. The family of continuous functions $\varphi_{\alpha}:X\to \C$  is called a \emph{moment function sequence of rank $r$}, if 
	\begin{equation}\label{Eq3}
		\varphi_{\alpha}(x*y)=\sum_{\beta\leq \alpha} \binom{\alpha}{\beta} \varphi_{\beta}(x)\varphi_{\alpha-\beta}(y)
	\end{equation}
holds whenever $x,y$ are in $X$ and $\alpha$ is in $\N^r$. We may consider finite sequences as well, if we restrict $|\alpha|\leq N$ with some nonnegative integer $N$. Clearly, the function $\varphi_0$ is an exponential, and we shall say that the moment function sequence {\it corresponds to the} {\hbox{exponential $\varphi_0$}.

We call a continuous function $\varphi:X\to \C$ a {\it moment function}, if there is a positive integer $r$, a moment sequence $(\varphi_{\alpha})_{\alpha \in \mathbb{N}^{r}}$  of rank $r$, and a multi-index $\alpha$ in $\N^r$ such that $\varphi=\varphi_{\alpha}$. In this case $\alpha$ is called the {\it multi-degree} of $\varphi$. 

Given the commutative hypergroup $X$  and an element $y$ in $X$, the operator $\tau_y$ defined by 
$$
\tau_yf=\delta_{\check{y}}*f
$$
for each $f$ in $\mathscr C(X)$ is called the {\it translation by $y$}. Given an exponential $m$ the 
{\it modified difference corresponding to $m$ with increment $y$} is the measure 
$$
\Delta_{m;y}=\delta_{\check{y}}-m(y)\delta_o.
$$
The corresponding convolution operator can be written as 
$$
\Delta_{m;y}*f=(\tau_y-m(y) \operatorname{id})f.
$$
A closed linear subspace in $\mathscr C(X)$ is called a {\it variety}, if it is invariant with respect all translation operators.

Products of modified differences corresponding to $f$ will be written in the form
$$
\Delta_{f;y_1,y_2,\dots,y_{n+1}}=\Delta_{f;y_1}*\Delta_{f;y_2}*\cdots*\Delta_{f;y_{n+1}}
$$
whenever $y_1,y_2,\dots,y_{n+1}$ are in $X$. It is clear that the continuous function \hbox{$m:X\to\C$} is an exponential if and only if $m(o)=1$, and 
$$
\Delta_{m;y}*m=0
$$
for each $y$ in $X$. 

Given an exponential $m$ on $X$ the continuous function $\varphi:X\to\C$ is called a {\it generalized $m$-exponential monomial} if there exists a natural number $n$ such that
$$
\Delta_{m;y_1,y_2,\dots,y_{n+1}}*\varphi=0
$$
whenever $y_1,y_2,\dots,y_{n+1}$ are in $X$. If $\varphi$ is nonzero, then the the exponential $m$ and the smallest $n$ with this property are uniquely determined, and $n$ is called the {\it degree} of $\varphi$. Clearly, $\varphi$ is of degree $n\geq 1$ if and only if $\Delta_{m;y}*\varphi$ is of degree $n-1$ for some $y$.  For instance, the $m$-exponential monomials of degree $0$ are exactly the constant multiples of $m$. The $m$-exponential monomials of degree $1$, which vanish at $o$ are called {\it $m$-sine functions}. A generalized $m$-exponential monomial is called an \emph{$m$-exponential monomial} if its variety is finite dimensional. 

An important property of moment functions is expressed in the following theorem.

\begin{thm}\label{momexp}
Let $X$ be a commutative hypergroup and $m$ an exponential on $X$. Every moment function corresponding to $m$ is an $m$-exponential monomial. 
\end{thm}
\begin{proof}
	We prove the statement by induction on $|\alpha|$, where $\alpha$ is the multi-degree of the moment function $\varphi$ of rank $r$, which corresponds to the exponential $m$. If $|\alpha|=0$, then $\alpha$ is the zero of $\N^r$ hence, by \eqref{Eq3}, we infer that $\varphi=m$ is an exponential. Assume that we have proved the statement for each $\alpha$ in $\N^r$ with $|\alpha|\leq N$, and let $\varphi$ be a moment function of rank $r$ corresponding to the exponential $m$, having multi-degree $\alpha$ with $|\alpha|=N+1$. From \eqref{Eq3} we infer
	$$
	\Delta_{m;y}\varphi_{\alpha}(x)=\varphi_{\alpha}(x*y)-m(y)\varphi_{\alpha}(x)=\sum_{0\leq \beta<\alpha} \binom{\alpha}{\beta} \varphi_{\beta}(x) \varphi_{\alpha-\beta}(y)
	$$
	for each $x,y$ in $X$. On the right side we have a linear combination of the moment functions $\varphi_{\beta}$, each of them having multi-degree $\beta<\alpha$, hence $|\beta|\leq N$. By the induction hypothesis, we have that 
	$$\Delta_{m;y_1,y_2,\dots,y_{N+1},y}\varphi_{\alpha}(x)=\Delta_{m;y_1,y_2,\dots,y_{N+1}}*\Delta_{m;y}\varphi_{\alpha}(x)
	=
	\sum_{0\leq \beta<\alpha} \binom{\alpha}{\beta} \Delta_{m;y_1,y_2,\dots,y_{N+1}}*\varphi_{\beta}(x)\cdot  \varphi_{\alpha-\beta}(y)=0,$$
	hence  $\varphi_{\alpha}$ is a generalized $m$-exponential monomial of degree at most $N+1$. By \eqref{Eq3}, the variety of every moment function is finite dimensional, as it is spanned by the finitely many functions $\varphi_{\beta}$ with $0\leq \beta\leq \alpha$, every moment function is an exponential monomial.
\end{proof}

In fact, as an additional information we have proved that the multi-degree $\alpha$ of a moment function $\varphi$ of rank $r$ satisfies $\deg \varphi\leq |\alpha|$, where $\deg \varphi$ denotes its degree, as an exponential monomial. We note that, in general, we do not have equality here: for instance, the zero function may have any nonzero multi-degree, as a moment function. 

Now we are going to give an example of a hypergroup, which is not a group, where generalized $m$-exponential monomials exist. Let us consider the hypergroup joins $C\vee D$. The construction and properties of a hypergroup joins can be found e.g. in \cite{SzeVat20}. We consider a particular case, where the compact part $C$ is the two-point hypergroup $D(\theta)=\left\{ 0,i \right\}$ and the discrete part $D$ is $\R$ endowed with the discrete topology. Using \cite[theorem. 7]{SzeVat20} we know the form of generalized exponential monomials on $D(\theta)\vee \R$. More precisely: a continuous function $f$ on $D(\theta) \cup \R$ is a generalized exponential monomial of degree at most $n$ if and only if one of the following cases holds:
\begin{enumerate}
	\item $f|_{D(\theta)}$ is a generalized exponential monomial of degree at most $n$ associated with an exponential $m|_{D(\theta)}\neq 1$ on $D(\theta)$  and $f|_{\R}=0$
	\item $f|_{D(\theta)}$ is constant and $f|_{\R}$ is a generalized exponential monomial on $\R$.
\end{enumerate} 

This means that there exist generalized exponential monomials on $D(\theta)\vee \R$, which are not exponential monomials, as on $\R$ there exist nontrivial generalized exponential monomial. For more detailed discussion see \cite{Sze13a} and references therein.

\section{Generalized derivations}

In our former paper \cite{FecGseSze22b} we investigated the relation between moment function sequences and higher order derivations on the measure algebra. We proved that there is a close connection between these two concepts, especially when we assume that the derivations are not just linear operators on the measure algebra, but they are also module homomorphisms if the measure algebra is considered as a module over the ring of continuous functions. As moment functions are special exponential monomials, it is reasonable to ask if it is possible to extend the concept of higher order derivations in a way such that this extension relates to generalized exponential monomials. We present a possible answer this question. 

Let $X$ be a commutative hypergroup and $r$ a positive integer. The measure algebra $\mathscr M_c(X)$ will be considered as a module over the ring $\mathscr C(X)$ of continuous complex valued functions on $X$ with the action
$$
\langle \varphi \cdot \mu,f\rangle=\int_X f\cdot \varphi\,d\mu. 
$$
We shall write simply $\varphi\mu$ for $\varphi\cdot \mu$.

We recall (see \cite{FecGseSze22b}) that the family of module homomorphisms $(D_{\alpha})_{\alpha\in \N^r}$ on $\mathscr M_c(X)$ is called a {\it higher order derivation of rank $r$}, or simply {\it higher order derivation}, if for each $\alpha$ in $\N^r$ we have
\begin{equation}\label{highor}
D_{\alpha}(\mu *\nu)=\sum_{\beta\leq \alpha} \binom{\alpha}{\beta}D_{\beta}\mu * D_{\alpha-\beta}\nu
\end{equation}
whenever $\mu,\nu$ are in $\mathscr M_c(X)$. We underline that the module homomorphism property means that each $D_{\alpha}$ satisfies the two functional equations
\begin{eqnarray*}
D_{\alpha}(\mu+\nu)&=&D_{\alpha}\mu+D_{\alpha}\nu\\
D_{\alpha} (\varphi \mu)&=&\varphi   D_{\alpha}\mu
\end{eqnarray*}
for each $\mu,\nu$ in $\mathscr M_c(X)$ and $\varphi$ in $\mathscr C(X)$. We say that this higher order derivation is {\it continuous}, if each operator $D_{\alpha}$ is continuous with respect to the weak*-topology on $\mathscr C(X)$. We note that $D_0$ is not just a module homomorphism, but it is also an algebra homomorphism, that is $D_0(\mu*\nu)=D_0\mu*D_0\nu$. We shall say that $D_0$ is a {\it multiplicative module homomorphism}, and that this higher order derivation {\it corresponds to  $D_0$}.

We note that our terminology is somewhat different from the usual one. Name\-ly, in general a linear operator $D:\mathscr M_c(X)\to \mathscr M_c(X)$ is called a {\it derivation}, if we have $D(\mu*\nu)=\mu* D\nu+D\mu*\nu$  for each $\mu,\nu$ in $\mathscr M_c(X)$. So, in the above definition, for the case $|\alpha|=1$ instead of homogeneity with respect to multiplication by complex numbers we require homogeneity with respect to multiplication by continuous functions, and instead of the identity operator $\operatorname{id}$ we use an arbitrary multiplicative module homomorphism $D_0$. This leads to the equation
$$
D(\mu*\nu)=D_0\mu* D\nu+D\mu*D_0\nu.
$$

Now we introduce a more general concept of derivations on $\mathscr M_c(X)$ recursively. Let $r$ a positive integer. For each $\alpha$ with $|\alpha|=1$ in $\N^r$, we say that the  module homomorphism $D:\mathscr M_c(X)\to \mathscr M_c(X)$ is a {\it generalized derivation of order $\alpha$}, if there exists a multiplicative module homomorphism $D_0$ of $\mathscr M_c(X)$  such that 
$$
D(\mu*\nu)=D_0\mu* D\nu+D\mu*D_0\nu
$$
for each $\mu,\nu$ in $\mathscr M_c(X)$. We say that $D$ {\it corresponds to $D_0$}. Assume that $\alpha$ is in $\N^r$, and we have defined generalized derivations corresponding to the multiplicative module homomorphism $D_0$ of order $\beta$ for each $\beta< \alpha$ in $\N^r$. Then the module homomorphism $D:\mathscr M_c(X)\to \mathscr M_c(X)$ is called a {\it generalized derivation of order $\alpha$} corresponding to the multiplicative module homomorphism $D_0$, if the two variable function 
$$
(\mu,\nu)\mapsto D(\mu*\nu)-D_0\mu* D\nu-D\mu* D_0\nu
$$
is a generalized derivation of order less than $\alpha$ corresponding to $D_0$  in both  variables, when the other variable is fixed. We extend this terminology for the case $\alpha=0$ by saying that any multiplicative module homomorphism is a generalized derivation of order zero -- corresponding to itself. We note that we may call these generalized derivations of rank $r$, if we want to underline the role of $r$ in the definition.

First we show that this concept is indeed a generalization of the concept of higher order derivations.

\begin{thm}\label{dergen}
	Let $X$ be a commutative hypergroup, $r$ a positive integer, and assume that the sequence of functions $(D_{\alpha})_{\alpha\in\N^r}$ is a higher order derivation of rank $r$ on $\mathscr M_c(X)$. Then $D_{\alpha}$ is a generalized derivation of order $\alpha$, for each $\alpha$ in $\N^r$. 
\end{thm}

\begin{proof}
We prove this statement by induction on $|\alpha|$. For $|\alpha|=1$,we assume that $D_{\alpha}\colon \mathscr M_c(X)\to \mathscr M_c(X)$ is a derivation of order $\alpha$ corresponding to the multiplicative module homomorphism $D_0$. This means that we have 
\[
 D_{\alpha}(\mu*\nu)= D_{0}\mu* D_{\alpha}\nu+D_{\alpha}\mu* D_{0}\nu
\]
for each $\mu,\nu$ in $\mathscr M_c(X)$. In other words, 
\[
 D_{\alpha}(\mu*\nu)-D_{0}\mu* D_{\alpha}\nu-D_{\alpha}\mu* D_{0}\nu=0 
 \]
for each $\mu,\nu$ in $\mathscr M_c(X)$, and this is obviously a generalized derivation of order zero in both variables. 

Assume now that there exists an $\alpha$ in $\mathbb{N}^{r}$ with the property that the statement holds for all multi-indices less then $\alpha$. Let further $D_{\alpha}\colon \mathscr M_c(X)\times \mathscr M_c(X)\to \mathscr M_c(X)$ be a generalized derivation of order $\alpha$. Then 
\[
 D_{\alpha}(\mu*\nu)= \sum_{\beta\leq \alpha}\binom{\alpha}{\beta}D_{\beta}\mu *D_{\alpha-\beta}\nu
\]
for each $\mu,\nu$ in $\mathscr M_c(X)$. After some rearrangement, we get that 
\[
 D_{\alpha}(\mu*\nu)-D_{0}\mu* D_{\alpha}\nu-D_{\alpha}\mu *D_{0}\nu= \sum_{0< \beta < \alpha}\binom{\alpha}{\beta}D_{\beta}\mu* D_{\alpha-\beta}\nu
 \]
for each $\mu,\nu$ in $\mathscr M_c(X)$. Here the right side is a symmetric mapping and, due to the induction hypothesis, it is a generalized derivation of order at most $\beta<\alpha$ in both variables. This proves our statement. 
\end{proof}

\begin{thm}
Let $X$ be a commutative hypergroup, $r$ a positive integer and $\alpha$ in $\N^r$. Given a continuous generalized derivation $D$ of order $\alpha$ on $\mathscr M_c(X)$ corresponding to the continuous multiplicative module homomorphism $D_0$  we define
\begin{equation}\label{mfidef}
m(x)=\langle D_0\delta_x,1\rangle
\qquad 
\text{ and }
\qquad
\varphi(x)=\langle D\delta_x,1\rangle
\end{equation}
for each $x$ in $X$. Then $m$ is an exponential, and $\varphi$ is a generalized $m$-exponential monomial of degree $|\alpha|$, further we have
\begin{equation}\label{mfidefre}
\langle D_0\mu,f\rangle=\int_X  f\cdot m\,d\mu
\qquad 
\text{and }
\qquad
\langle D\mu,f\rangle=\int_X  f\cdot \varphi\,d\mu
\end{equation}
for each $\mu$ in $\mathscr M_c(X)$ and $f$ in $\mathscr C(X)$.
\end{thm}

\begin{proof}
The continuity of $m$ and $\varphi$ follows easily from the continuity of $D_0$ and $D$. For each $x,y$ in $X$ we have
\begin{multline*}
m(x*y)=\langle D_0\delta_{x*y},1\rangle=\int_X 1 \,dD_0\delta_{x*y}=\int_X 1 \,dD_0(\delta_{x}*\delta_y)
\\
=
\int_X \int_X 1 \,dD_0\delta_{x}\,dD_0\delta_{y}=
\int_X 1 \,dD_0\delta_{x}\cdot \int_X 1 \,dD_0\delta_{y}=m(x)m(y).
\end{multline*}
As $D_0$ is nonzero, we have $m(o)=1$, hence $m$ is an exponential.
\vskip.2cm

The statement about $\varphi$ will be proved by induction on $|\alpha|$, and it is true for $|\alpha|=0$. Suppose that the statement holds for $|\alpha|\leq N$, and now we prove it for an arbitrary multi-index $\alpha'$ with $|\alpha'|=N+1$. Let $x,y,y_1,y_2,\dots,y_{N+1}$ be any elements in $X$, and we show that 
$$
\Delta_{m;y_1,y_2,\dots,y_{N+1},y}*\varphi(x)=\Delta_{m;y_1,y_2,\dots,y_{N+1}}*\Delta_{m;y}*\varphi(x)=0.
$$
We let $\psi(x)=\varphi(x*y)-m(y)\varphi(x)$, then it is enough to show that 
$$
\Delta_{m;y_1,y_2,\dots,y_{N+1}}*\psi(x)=0.
$$
We introduce the notation
$$
B(\mu,\nu)=D(\mu*\nu)-D_0\mu*D\nu-D\mu*D_0\nu
$$
whenever $\mu,\nu$ are in $\mathscr M_c(X)$. Then $B$ is a generalized derivation of order at most $\alpha$ in both variables, and we can compute as follows:
\begin{flalign*}
\psi(x)&=\varphi(x*y) -m(y)\varphi(x)=\langle D\delta_{x*y}-D_0\delta_y* D\delta_x,1\rangle&
\\
&=\langle D\delta_{x}*D\delta_{y}-D_0\delta_y* D\delta_x,1\rangle =\langle D\delta_y* D_0\delta_x,1\rangle+\langle B(\delta_x,\delta_y),1\rangle&
\\
&=\int_X 1 \,d(D\delta_y* D_0\delta_x)+\langle B(\delta_x,\delta_y),1\rangle
=\int_X \int_X 1 \,dD\delta_y\, dD_0\delta_x+\langle B(\delta_x,\delta_y),1\rangle &
\\
&=\int_X 1 \,dD\delta_y \int_X 1 \, dD_0\delta_x+\langle B(\delta_x,\delta_y),1\rangle
=\varphi(y)m(x)+\langle B(\delta_x,\delta_y),1\rangle.&
\end{flalign*}
The function $x\longmapsto \langle B(\delta_x,\delta_y),1\rangle$ is a generalized exponential monomial of degree less than $|\alpha'|$, hence its degree is at most $N$. It follows that $\psi$ is a generalized exponential monomial of degree less than $N$. We conclude that 
$$
\Delta_{m;y_1,y_2,\dots,y_{N+1},y}*\varphi(x)=\Delta_{m;y_1,y_2,\dots,y_{N+1}}*\psi(x)=0,
$$
which proves that $\varphi$ a generalized exponential monomial of degree $N+1$.
\vskip.2cm

To prove \eqref{mfidefre}, observe that the mapping $\mu\mapsto \langle D_0\mu,1\rangle$ is a weak*-contin\-u\-ous linear functional on $\mathscr M_c(X)$, hence it arises from a function $\varphi_0$ in $\mathscr C(X)$ in the following way:
$$
\langle D_0\mu,1\rangle=\int_X \varphi_0\,d\mu
$$ 
for each $\mu$ in $\mathscr M_c(X)$. For $\mu=\delta_x$ this gives
$$
\varphi_0(x)=\int_X \varphi\,d\delta_x=\langle D_0\delta_x,1\rangle=m(x).
$$
Similarly, we obtain the second part of \eqref{mfidefre}, and the proof is complete.
\end{proof}

\begin{prop}
 Let $X$ be a commutative hypergroup and $F\colon \mathscr M_c(X) \to \mathscr{M}_{c}(X)$ be a continuous module homomorphism such that $\factor{\mathscr M_c (X)}{\operatorname{ker}(F)}$ is finite dimensional and define the continuous function 
 $f\colon X\to \mathbb{C}$ by 
 \[
  f(x)= \langle F\delta_{x}, 1\rangle
 \]
for each $x\in X$. Then the variety $\tau(f)$  is finite dimensional. 
\end{prop}

 \begin{proof}
  Suppose that we are given a continuous module homomorphism
  $F\colon \mathscr M_c(X) \allowbreak\to \mathscr M_c(X)$ such that $\factor{\mathscr M_c(X)}{\ker(F)}$ is finite dimensional. Due to the Fundamental Theorem of Homomorphisms, we immediately get that the range of $F$ is of finite dimension, since we have $\mathrm{rng}(F)\simeq \factor{\mathscr M_c(X)}{\operatorname{ker}(F)}$. 

Thus there exists a positive integer $n$ and there exists continuous functions $F_{i}\colon \mathscr{M}_{c}(X)\to \mathscr{M}_{c}(X)$ such that 
  \[
   F(\mu)=\sum_{i=1}^{n}F_{i}(\mu)
  \]
 for all $\mu$ in $\mathscr{M}_{c}(X)$. 
 Especially, for all $x\in X$ we have 
\[
   F(\delta_{x})=\sum_{i=1}^{n}F_{i}(\delta_{x})
\]
and also for all $x, y$ in $X$
\[
 F(\delta_{x}\ast \delta_{y})= \sum_{i=1}^{n}G_{i}(\delta_{y})F_{i}(\delta_{x}). 
\]
 
 Consider the function $f\in \mathscr{C}(X)$ defined by 
 \[
  f(x)=\langle F \delta_{x}, 1\rangle 
  \qquad 
  \left(x\in X\right).
 \]
 Then for all $x, y$ in $X$ we have 
 \begin{align*}
  \tau_{y}\ast f(x)
  = 
  f(x\ast y)
  &=
  \langle F(\delta_{x}\ast \delta_{y}), 1\rangle
  =
  \langle \sum_{i=1}^{n}F_{i}(\delta_{x})\ast G_{i}(\delta_{y}), 1\rangle
  \\
  &=
  \sum_{i=1}^{n} \langle F_{i}(\delta_{x})\ast G_{i}(\delta_{y}), 1\rangle
  =
  \sum_{i=1}^{n}\int_{X}1d\left(F_{i}(\delta_{x})\ast G_{i}(\delta_{y})\right)
  \\
  &=
  \sum_{i=1}^{n}\int_{X}\int_{X}1dF_{i}(\delta_{x})dG_{i}(\delta_{y})
  =
  \sum_{i=1}^{n}\int_{X} 1dF_{i}(\delta_{x}) \cdot \int_{X}1dG_{i}(\delta_y)
  \\
  &= 
  \sum_{i=1}^{n}f_{i}(x)g_{i}(y), 
  \end{align*}
 where the continuous functions $f_{i}, g_{i}\colon X\to \mathbb{C}$, $i=1, \ldots, n$  are defined through 
\[
  f_{i}(x)=\langle F_{i}\delta_{x}, 1\rangle= \int_{X} 1dF_{i}(\delta_{x})  
  \quad 
  \text{and}
  \quad 
  g_{i}(x)=\langle G_{i}\delta_{x}, 1\rangle= \int_{X} 1dG_{i}(\delta_{x})  
 \]
 for all $x$ in $X$. This means however that all the translates of the function $f$ belong to a finite dimensional linear space. So the variety $\tau(f)$ is finite dimensional. 
 \end{proof}

\begin{thm}
 Let $X$ be a commutative hypergroup, $r$ a positive integer and $\alpha$ in $\N^r$. Given a continuous generalized derivation $F$ of order $\alpha$ on $\mathscr M_c(X)$ corresponding to the continuous multiplicative module homomorphism $D_0$, the following statements are equivalent. 
 \begin{enumerate}[(i)]
  \item $\factor{\mathscr{M}_{c}(X)}{\ker(F)}$ is finite dimensional. 
  \item The mapping $F$ is a higher order derivation of rank $r$ with degree $\alpha$. 
 \end{enumerate}
\end{thm}

\begin{proof}
 Firstly we prove the direction $(ii)\Rightarrow (i)$. 
 Accordingly, let $X$ be a commutative hypergroup, $r$ a positive integer and $\alpha$ in $\N^r$. Assume further that $F$ is a higher order derivation of rank $r$ with degree $\alpha$ on $\mathscr{M}_{c}(X)$ corresponding to the continuous multiplicative module homomorphism $D_0$. 
 Then for all $\beta\in \mathbb{N}^{r}$ there exists a continuous module homomorphism 
 $D_{\beta}$ on $\mathscr{M}_{c}(X)$ such that $F= D_{\alpha}$ and for all $\beta \in \mathbb{N}^{r}$ we have 
 \[
  D_{\beta}(\mu \ast \nu)= \sum_{\gamma \leq \beta}\binom{\beta}{\gamma}D_{\gamma}(\mu)\ast D_{\beta-\gamma}(\nu)
 \]
for all $\mu, \nu$ in $\mathscr{M}_{c}(X)$. 
Especially, we have that for all $\mu\in \mathscr{M}_{c}(X)$
\[
F(\mu)=
 D_{\alpha}(\mu)= 
 D_{\alpha}(\mu \ast \delta_{o})= \sum_{\beta\leq \alpha}\binom{\alpha}{\beta}D_{\beta}(\mu)\ast D_{\alpha-\beta}(\delta_{o}) 
\]
holds. In other words, 
\[
 \mathrm{rng}(F)\subseteq \mathrm{lin} \left\{D_{\beta}\, \vert \, \beta \leq \alpha \right\}. 
\]
At the same time, due to the Fundamental Theorem of Homomorphisms
\[\mathrm{rng}(F)\simeq \factor{\mathscr{M}_{c}(F)}{\ker(F)}.\]
Further the linear space $\mathrm{lin} \left\{D_{\beta}\, \vert \, \beta \leq \alpha \right\}$ is obviously of finite dimension. \\
Thus $\factor{\mathscr{M}_{c}(F)}{\ker(F)}$ is finite dimensional. 

Finally, we consider the direction $(i) \Rightarrow (ii)$. 
Suppose now that $X$ is a commutative hypergroup, $r$ a positive integer and $\alpha$ in $\N^r$. Let further $F$ be a continuous generalized derivation  of order $\alpha$ on $\mathscr M_c(X)$ corresponding to the continuous multiplicative module homomorphism $D_0$ so that $\factor{\mathscr{M}_{c}(X)}{\ker(F)}$ is finite dimensional. 
We prove the statement by induction on the height of the multi-index $\alpha$. If $|\alpha|=1$, then we have 
\[
 F(\mu \ast \nu)= D_{0}(\mu)\ast F(\nu)+F(\mu)\ast D_{0}(\nu)
\]
for all $\mu, \nu\in \mathscr{M}_{c}(X)$ showing that the statement holds trivially in case $|\alpha|=1$. 
Assume now that there exists a multi-index $\alpha$ such that the statement holds for all multi-indices $\beta$ that are (strictly) less then $\alpha$. 
In this case the mapping $D$ defined on $\mathscr{M}_{c}(X)\times \mathscr{M}_{c}(X)$ by 
\[
 D(\mu, \nu)= F(\mu \ast \nu)-D_{0}(\mu)\ast F(\nu)-F(\mu)\ast D_{0}(\nu) 
 \qquad 
 \left(\mu, \nu\in \mathscr{M}_{c}(X)\right)
\]
is symmetric and it is a generalized derivation of degree (strictly) less than $\alpha$ in each of its variables. 

On the other hand, due to the Fundamental Theorem of Homomorphisms, we have 
\[
 \factor{\mathscr{M}_{c}(X)}{\ker(F)} \simeq \mathrm{rng}(F) 
\]
and due to condition (i), $\factor{\mathscr{M}_{c}(X)}{\ker(F)}$ or equivalently, $\mathrm{rng}(F)$ is finite dimensional. Thus for all fixed $\nu^{\ast}\in \mathscr{M}_{c}(X)$, the values of the mapping 
\[
 \mathscr{M}_{c}(X)\ni \mu \longmapsto D(\mu, \nu^{\ast}) 
\]
belong to this finite dimensional algebra. Similarly, for all fixed $\mu^{\ast}\in \mathscr{M}_{c}(X)$, the values of the mapping 
\[
 \mathscr{M}_{c}(X)\ni \nu \longmapsto D(\mu^{\ast}, \nu) 
\]
belong to this finite dimensional algebra. Therefore there exists a multi-index $\widetilde{\alpha}\in \mathbb{N}^{r}$  and there exist selfmappings $F_{\beta}, G_{\beta}$ of $\mathscr{M}_{c}(X)$ for all multi-index $\beta \leq \widetilde{\alpha}$ such that 
\[
 D(\mu, \nu)= \sum_{\beta\leq \widetilde{\alpha}} F_{\beta}(\mu)\ast G_{\beta}(\nu)
\]
holds for all $\mu, \nu$ in $\mathscr{M}_{c}(X)$. Due to the induction hypothesis, this mapping is also a symmetric function which is a generalized derivation of degree (strictly) less than $\alpha$ in each of its variables. However, this is only possible if 
\[
 D(\mu, \nu)= \sum_{\beta < \alpha} \binom{\alpha}{\beta} D_{\beta}(\mu)\ast D_{\alpha-\beta}(\nu)
\]
for all $\mu, \nu$ in $\mathscr{M}_{c}(X)$, with an appropriate higher order derivation $(D_{\alpha})_{\alpha\in \mathbb{N}^{r}}$. 
\end{proof}

\begin{thm}\label{expder}
	Let $X$ be a commutative hypergroup, $n$ a natural number, $m$ an exponential, and $\varphi$ a generalized exponential monomial of degree $n$. Then there exists a continuous multiplicative module homomorphism $D_0$, and a continuous generalized derivation $D$ of order $n$ corresponding to $D_0$ such that
	\eqref{mfidef} holds for each $x$ in $X$.
\end{thm}

\begin{proof}
We define $D_0$ and $D$ in the obvious way, as given in \eqref{mfidefre}. Then it is a routine calculation to show that $\mu\longmapsto \langle D_0\mu,1\rangle$ is a multiplicative linear functional of the measure algebra $\mathscr M_c(X)$. Hence $D_0$ is a continuous multiplicative module homomorphism of $\mathscr M_c(X)$. Then the first equation of \eqref{mfidef} follows. We prove the second equation of \eqref{mfidef} by induction on $n$. For $n=0$ the statement is equivalent to the first equation in \eqref{mfidef}. Now we assume $n\geq 1$ and we suppose that the second equation of \eqref{mfidef} defines a continuous generalized derivation of degree $k$  for each generalized $m$-exponential monomial $\varphi$ of degree $k$, whenever $k=0,1,\dots,n$.  Now let $\varphi$ be a generalized $m$-exponential monomial $\varphi$ of degree $n+1$, and we define
$$
\langle B(\mu,\nu),1\rangle=\int_X \int_X (\varphi(x*y)-m(y)\varphi(x)-m(x)\varphi(y))\,d\mu(x)\,d\nu(y)
$$
whenever $\mu,\nu$ are in $\mathscr M_c(X)$. Clearly, $B:\mathscr M_c(X)\times \mathscr M_c(X)\to\mathscr M_c(X)$ is a continuous module homomorphism in both variables -- in fact, it is a continuous generalized derivation of order at most $n$, by the induction hypothesis, as the integrand is a generalized $m$-exponential monomial of degree at most $n$, in both variables. On the other hand, for each $\mu,\nu$ in $\mathscr M_c(X)$ we have:
$$
\langle D(\mu*\nu)-D_0\mu*D\nu-D\mu*D_0\nu,1\rangle=\langle B(\mu,\nu),1\rangle.
$$
Using the homogeneity with respect to multiplication by continuous functions of all these operators in their arguments we get
$$
\langle D(\mu*\nu)-D_0\mu*D\nu-D\mu*D_0\nu,f\rangle=\langle B(\mu,\nu),f\rangle.
$$
for each $f$ in $\mathscr C(X)$. This proves that $D:\mathscr M_c(X)\to\mathscr M_c(X)$ is a continuous generalized derivation of order at $n+1$, and our proof is complete.
\end{proof}

\begin{ackn}
(E. Gselmann and L. Sz{\'e}kelyhidi): Project No. K134191 has been implemented with the support provided by the National Research,
Development and Innovation Fund of Hungary, financed under the K20 funding scheme. 
\end{ackn}


\begin{thebibliography}{PTW02}
\bibitem{BloHey95}
 W. R. Bloom, H. Heyer,
\newblock \emph{Harmonic analysis of probability measures on hypergroups},
  volume~20 of (De Gruyter Studies in Mathematics, Walter de Gruyter \& Co., Berlin, 1995)

\bibitem{FecGseSze22b}
\.Z. Fechner, E. Gselmann, L. Sz\'{e}kelyhidi, {\rm Endomorphisms of the measure algebra of commutative hypergroups}, (2022), \href{https://arxiv.org/abs/2204.07499}{arXiv:2204.07499}

\bibitem{Sze13}
L. Sz{\'e}kelyhidi, 
\newblock \emph{Functional equations on hypergroups}.
(World Scientific Publishing Co. Pte. Ltd., Hackensack, NJ, 2013)

\bibitem{Sze13a} L. Sz\'ekelyhidi, \emph{Exponential polynomials on commutative hypergroups}, Arch. Math. 101, 341–347 (2013). 

\bibitem{MR3185617}
L. Sz{\'e}kelyhidi,
\newblock \emph{Harmonic and spectral analysis}.
(World Scientific Publishing Co. Pte. Ltd., Hackensack, NJ, 2014)


\bibitem{SzeVat20} K. Vati, L. Sz\'ekelyhidi, {\rm Exponential monomials on hypergroup joins}, Miskolc Mathematical Notes {\bf 21.1} (2020), 463-472.
\end{thebibliography}
\end{document}